\documentclass[a4paper,twoside,11pt,leqno]{article}
\usepackage{fullpage}
\usepackage{amsmath} \usepackage{amssymb} \usepackage{amsopn}
\usepackage{amsthm} \usepackage{amsfonts} \usepackage{mathrsfs}
\usepackage{verbatim}

\usepackage[OT2,T1]{fontenc}

\usepackage[pdfpagelabels,pdftex]{hyperref}
\usepackage[dvips]{graphicx} \usepackage[usenames]{color}
\usepackage[all]{xy}

\usepackage{setspace}
\usepackage{amssymb}
\usepackage{euscript}
\usepackage{cite}
\usepackage[all]{xy}
\usepackage{tabularx}

\theoremstyle{plain}
\usepackage{amsfonts}
\usepackage{graphicx}
\usepackage{amsmath}

\hypersetup{
  pdftitle={},
  pdfauthor={},
  pdfsubject={},
  pdfkeywords={},
  colorlinks=false,
  breaklinks=true,
  bookmarksopen=true,
  bookmarksnumbered=true,
  pdfpagemode=UseOutlines,
  plainpages=false}

\newcommand{\bC}{{\mathbb C}}

\newcommand{\bF}{{\mathbb F}}

\newcommand{\bH}{{\mathbb H}}

\newcommand{\bQ}{{\mathbb Q}}
\newcommand{\bR}{{\mathbb R}}

\newcommand{\bZ}{{\mathbb Z}}

\newcommand{\cL}{{\mathscr L}}

\newcommand{\caO}{{\mathcal O}}

\newcommand{\fP}{{\mathfrak P}}
\newcommand{\fQ}{{\mathfrak Q}}

\newcommand{\fc}{{\mathfrak c}}

\newcommand{\fm}{{\mathfrak m}}

\newcommand{\fp}{{\mathfrak p}}

\DeclareSymbolFont{cyrletters}{OT2}{wncyr}{m}{n}
\DeclareMathSymbol{\Sha}{\mathalpha}{cyrletters}{"58}

\DeclareMathOperator{\pr}{pr}

\DeclareMathOperator{\Hom}{Hom}

\DeclareMathOperator{\coker}{coker}

\DeclareMathOperator{\Spec}{Spec}

\DeclareMathOperator{\Frob}{Frob}

\DeclareMathOperator{\res}{res}
\DeclareMathOperator{\cores}{cor}

\DeclareMathOperator{\Gal}{G}

\newcommand{\nr}{{\rm nr}}

\makeatletter
\newtheorem*{rep@theorem}{\rep@title}
\newcommand{\newreptheorem}[2]{%
\newenvironment{rep#1}[1]{%
 \def\rep@title{#2 \ref{##1}}%
 \begin{rep@theorem}}%
 {\end{rep@theorem}}}
\makeatother

\newtheorem{thm}{Theorem}[section]

\newtheorem{prop}[thm]{Proposition}

\newtheorem{cor}[thm]{Corollary}

\newtheorem{lm}[thm]{Lemma}

\newtheorem{claim}[thm]{Claim}

\newreptheorem{thm}{Theorem}
\newreptheorem{prop}{Proposition}
\newreptheorem{cor}{Corollary}

\theoremstyle{definition}
\newtheorem{Def}[thm]{Definition}

\newtheorem{rem}[thm]{Remark}
\newtheorem{rems}[thm]{Remarks}

\newenvironment{pro*}[1][Proof]{{\it{#1:}} }{}

\newcommand{\abs}{\sharp}

\newcommand\rar{ \rightarrow }
\newcommand\tar{ \twoheadrightarrow }
\newcommand\har{ \hookrightarrow }

\newcommand\ord{\mathop{\rm ord}}

\newcommand\cs{\mathop{ \rm cs}}

\newcommand\dirlim{\mathop{\underrightarrow{\lim} }}

\newcommand{\sm}{{\,\smallsetminus\,}}
\newcommand\N{\mathop{\rm N}}

\DeclareMathOperator{\coh}{H}

\newcommand\subsetsim{\stackrel{\subset}{\sim}}
\newcommand\supsetsim{\stackrel{\supset}{\sim}}

\DeclareMathOperator{\Ram}{Ram}

\DeclareMathOperator{\lcm}{lcm}

\newcounter{absatzcounter}[section]
\numberwithin{equation}{section}

\begin{document}

\title{Densities of primes and realization of local extensions}
\author{A. Ivanov}
% \email{ivanov@mathi.uni-heidelberg.de}
% \address{Mathematisches Institut, Universit\"at Heidelberg, Im Neuenheimer Feld 288, 69120 Heidelberg, Germany}
% \classification{11R34, 11R45}
% \keywords{Number Field, Galois Cohomology, Restricted Ramification, Dirichlet Density}
% \thanks{The author was supported by the Mathematical Center Heidelberg}
\maketitle

\begin{abstract}
In this paper we introduce new densities on the set of primes of a number field. If $K/K_0$ is a Galois extension of number fields, we associate to any element $x \in \Gal_{K/K_0}$ a density $\delta_{K/K_0,x}$ on primes of $K$. In particular, the density associated to $x = 1$ is the usual Dirichlet density on $K$. Using these densities (for $x \neq 1$) we prove realization results \`a la Grunwald-Wang theorem such that essentially, ramification is only allowed in a set of primes of density zero. 

\end{abstract}

\section{Introduction}

In this paper we adress the question of generalizing the Dirichlet density on the set of primes of a number field. In particular, we provide sets of primes with Dirichlet density zero with an appropriate positive measure. The second goal of the paper is to use these generalized densities to show a realization result of local extensions by global ones satisfying certain conditions.

\begin{comment}
Let $K$ be a number field, i.e., a finite extension of $\bQ$. This paper has two goals:
\begin{itemize}
\item[(1)] The first goal is to introduce (and establish the basic properties of) some new densities on the set of all primes of $K$. More precisely, to a subextension $K/K_0/\bQ$ with $K/K_0$ Galois and an element $x \in \Gal_{K/K_0}$, we will associate a density $\delta_{K/K_0,x}$ on primes of $K$. In particular, $\delta_{K/K_0,1}$ will be the usual Dirichlet density. Their use is that one can now apply density arguments to sets of primes with Dirichlet density $0$. Such sets actually occur also when one study sets with positive Dirichlet density, since this last can become zero after a finite extension.
\item[(2)] The second goal is to apply these generalized densities to prove the following global realization result: if $S$ contains an almost Chebotarev set and $R \subseteq S$ is finite, then for any $\fp \in S \sm R$, the the $\fp$-adic completion of the maximal extension $K_S^R$ of $K$, which is unramified outside $S$ and completely split in $R$, is the algebraic closure of $K_{\fp}$ and for $\fp \not\in S$, the $\fp$-adic completion of $K_S^R$ is the unramified closure of $K_{\fp}$. 
\end{itemize}
\end{comment}

%************************************************************************************************************************************************************
%************************************************************************************************************************************************************

\subsection*{Generalized densities}

Let $K/K_0$ be a finite Galois extension of number fields, i.e., of finite extensions of $\bQ$. Let $x \in \Gal_{K/K_0}$ be of order $d$. Let $P_{K/K_0}^x$ denote the set of all primes $\fp$ of $K$ which are unramified in $K/K_0$ and satisfy $\Frob_{\fp,K/K_0} = x$. We will introduce a density $\delta_{K/K_0,x}$ of a set $S$ of primes of $K$, which measures how big the ratio of the sizes of $S \cap P_{K/K_0}^x$ and $P_{K/K_0}^x$ is. This is done in the same way as for Dirichlet density, with the only difference that one has to take the limit over the ratio of terms of the kind $\sum_{\fp \in \ast} \N\fp^{-s}$ not over $s \rar 1$ but over $s \rar d^{-1}$ with $s$ lying in the right half plane $\Re(s) > d^{-1}$. Further, $\delta_{K/K_0,x}$ is essentially independent of the base field $K_0$, so one also could replace $K_0$ once for all time by $\bQ$, but it is easier to work with a \emph{Galois} extension $K/K_0$.

Once introduced, the most interesting thing about such a density is its base change behavior. To explain it, let $L/K$ be an extension such that $L/K_0$ is Galois. Write $H := \Gal_{L/K} \triangleleft \Gal_{L/K_0} =: G$ and $\pi \colon G \tar G/H$ for the natural projection. For any $y \in \pi^{-1}(x)$ we have the map induced by restriction of primes $P_{L/K_0}^y \rar P_{K/K_0}^x$. It is in general neither injective nor surjective. For $y,z \in \pi^{-1}(x)$ one easily sees that the images of the corresponding maps are either equal or disjoint and that the first is equivalent to $y,z$ being $H$-conjugate (cf. Lemma \ref{lm:H-conj_rel}). If $C$ is an $H$-conjugacy class in $\pi^{-1}(x)$, let $M_C$ denote the image of $P^y_{L/K_0}$ for some (any) $y \in C$ in $P_{K/K_0}^x$. We will show the following generalization of Chebotarev's density theorem and then obtain a description of the base change behavior of $\delta_{K/K_0,x}$ as a direct corollary:

\begin{repprop}{prop:verallgCheb}
Let $L/K/K_0,\pi,x$ be as above. Let $C$ be an $H$-conjugacy class in $\pi^{-1}(x)$. Then 
\[ \delta_{K/K_0,x}(M_C) = \frac{\abs{C}}{\abs{H}}. \]
\end{repprop}

\begin{repcor}{cor:bc_of_delta_x}
Let $y \in \pi^{-1}(x)$ and let $C$ be its $H$-conjugacy class in $\pi^{-1}(x)$. Then 
\[ \delta_{L/K_0,y}(S_L) = \frac{\abs{H}}{\abs{C}} \delta_{K/K_0,x}(S \cap M_C) \]
\noindent if both densities exist.
\end{repcor}

More general, for any function $\psi \colon \Gal_{K/K_0} \rar \bC$  one can define a weighted function by $\delta_{K/K_0, \psi}(S) := [K:K_0]^{-1}\sum_{x \in \Gal_{K/K_0}} \psi(x) \delta_{K/K_0,x}(S)$. Then for example the Dirichlet density is associated with the character of the regular representation of $G$.

Similarly as Serre extended the Dirichlet density to a density on the set of closed points of a scheme of finite type over $\Spec \bZ$, also the densities associated to fixed Frobenius elements should generalize in this way. Furthermore, it would be intereting to know, whether in the case of varieties of dimension $\geq 2$ over a perfect field, it is possible to define such fixed Frobenius densities for divisors (i.e, to non-closed points) as was done with the Dirichlet density by Holschbach \cite{Ho}.

Below we discuss a realization result, which proof uses the generalized densities. A further application of them concerns saturated sets introduced by Wingberg in \cite{Wi2} and will be discussed elsewhere.

Finally, we have an obsevation concerning $L$-functions: there is the following problem about extending $L$-functions in the same way as the densities above. Let $K/\bQ$ be a finite Galois extension and $x \in \Gal_{K/\bQ}$. Consider the following product associated to $x$ and a Dirichlet character $\chi$ modulo $\fm$:

\begin{equation} \label{eq:L_support_ala_Hecke}
L_x(\fm, s, \chi) := \prod_{\fp \in P_{K/\bQ}^x} \frac{1}{1 - \chi(\fp)\N\fp^{-s}}.
\end{equation}

\noindent This product converges on the right half plane $\Re(s) > d^{-1}$, where $d$ is the order of $x$. But in general this function has no analytic continuation to the whole complex plain (not even to the right half plane $\Re(s) > 0$). The reason is easy: let $\fm = 1, \chi = 1$. For $s \rar d^{-1}$ this product behaves like $d^{-\frac{1}{d}} (\frac{1}{s - d^{-1}})^{\frac{1}{d}}$, i.e., their difference is bounded for $s \rar d^{-1}$, and this last function clearly has no analytic continuation. A natural question is, whether this problem can be resolved, for example by taking the $d$-th power of the product above or by removing a half-line starting at $d^{-1}$ from the complex plain. Luckily, one does not need any non-vanishing results on such L-functions to show Proposition \ref{prop:verallgCheb}, as it follows by simple counting arguments from Chebotarev's density theorem.

% One application of these generalized densities is the proof of the main result in \cite{Iv3}, in which, after trivializing $\mu_p$, sets having Dirichlet density $0$ occur. However, it seems that density arguments involving $\delta_{K/K_0,x}$ with $x \neq 1$ are more subtle than in the case $x = 1$, simply because they involve more group theory of occuring Galois groups. This is the reason, why the proof in \cite{Iv3} is much more technical as one could expect from the ``good'' case, which uses Dirichlet density only.

% A further application is a remark on saturated sets, which can be found in Section \ref{sec:satur_sets}.

%************************************************************************************************************************************************************
%************************************************************************************************************************************************************

\subsection*{Realization of local extensions}

Let us first fix some notations. Let $\fc$ be a full class of finite groups (in the sense of \cite{NSW} 3.5.2). Let $R \subseteq S$ be two sets of primes of a number field $K$. Then $K_S^R(\fc)$ denotes the maximal pro-$\fc$-extension of $K$ which is unramified outside $S$ and completely split in $R$. Moreover, for a prime $\fp$ of $K$ we denote by $K_{\fp}(\fc)$ the maximal pro-$\fc$-extension of $K_{\fp}$ and by $K_{\fp}^{\nr}$ the maximal unramified extension of $K_{\fp}$.

For $\ell$ a rational prime or $\infty$, let $\fc_{\leq \ell}$ denote the smallest full class of all finite groups, containing the groups $\bZ/p\bZ$ for all $p \leq \ell$. Our main result will be the following generalization of \cite{NSW} 9.4.3, which handles the case of $\delta_K(S) = 1$.

% For a Galois extension $\cL/K$ and a prime $\fp$ of $K$, choose some fixed extension $\bap$ of $\fp$ to $\cL$. We write $\cL_{\fp}$ for the union of all completions $L_{\bap|_L}$, taken over all finite subextensions $\cL/L/K$. Then the field $\cL_{\fp}$ is henselian and its completion is equal to the $\bap$-adic completion of $\cL$. Further, for two sets $S \supseteq R$ of primes of $K$, let $K_S^R$ denote the maximal pro-solvable extension of $K$, which is unramified outside $S$ and completely split in $R$. Moreover, for $\fc$ a full class of groups, we denote by $L(\fc)/K$ the maximal $\fc$-subextension of a Galois $L/K$. For $p$ a rational prime or $\infty$, let $\fc_{\leq p}$ denote the smallest full class of all finite groups, containing the groups $\bZ/\ell\bZ$ for all $\ell \leq p$. Our main result will be the following theorem.

\begin{thm}\label{thm:real_of_loc_ext}
Let $K$ be a number field, $S \supseteq R$ sets of primes of $K$, such that $R$ is finite and $S \supsetsim P_{M/K}(\sigma)$ for some finite extension $M/K$ and $\sigma \in \Gal_{M/K}$. For any $\ell \leq \infty$ and any prime $\fp$ of $K$ we have:

\begin{equation*}
(K_S^R(\fc_{\leq \ell}))_{\fp} = \begin{cases} K_{\fp}(\fc_{\leq \ell}) & \text{if } \fp \in S \sm R \\ K_{\fp}(\fc_{\leq \ell}) \cap K_{\fp}^{\nr} & \text{if } \fp \not\in S \\ K_{\fp} & \text{if } \fp \in R. \end{cases}
\end{equation*}
\noindent In particular, since absolute Galois groups of local fields are solvable, taking $\ell = \infty$ shows that the maximal solvable subextension of $K_S^R/K$ lies dense in $\overline{K_{\fp}}$ resp. in $K_{\fp}^{\nr}$ for $\fp \in S \sm R$ resp. $\fp \not\in S$.
\end{thm}

One part of the proof of Theorem \ref{thm:real_of_loc_ext}, namely to realize a $p$-extension with given local properties, when $S$ satisfies the property $(\dagger)_p$ introduced in \cite{IvStableSets} (see also Section \ref{sec:complem_stable_sets} below) was already done in \cite{IvStableSets}. Essentially, $(\dagger)_p$ means that $S$ contains many primes $\fp$, which are completely split in $K(\mu_p)/K$. The remaining and much more delicate case is when $\delta_{K(\mu_p)}(S_{K(\mu_p)}) = 0$ holds. Then the usual methods from \cite{NSW} and \cite{IvStableSets} do not apply anymore. Moreover, in such a case the pro-$p$-version of the theorem easily can fail.  For example, suppose that $\mu_p \not\subseteq K$ and $K(\mu_p)/K$ is totally ramified at each $p$-adic prime, let $1 \neq \sigma \in \Gal_{K(\mu_p)/K}$ and set $S := P_{K(\mu_p)/K}(\sigma)$. Then any prime $\fp \in S$ is unramified in $K_S(p)/K$, as $\fp \not\in S_p$ and $\mu_p \not\subseteq K_{\fp}$. Hence $K_S(p) = K_{\emptyset}(p)$. In particular,
 let $K = \bQ$ and $p$ odd. Then $\bQ_S(p) = \bQ_{\emptyset}(p) = \bQ$, i.e., the maximal possible local $p$-extension is realized nowhere. 

However, in the pro-$\fc_{\leq \ell}$-case the theorem holds. For example take in the above example $\ell = 3$. The set $S := P_{\bQ(\mu_3)/\bQ}(\sigma)$ satisfies $(\dagger)_p$ for all $p \neq 3$, and in particular $(\dagger)_2$. Hence at any $\fp \in S$ the maximal pro-$2$-extension can be realized, and hence $\mu_3 \subseteq \bQ_{S,\fp}$. After going up to an appropriate finite subextension $\bQ_S(\fc_{\leq 2})/K/\bQ$, the set $P_{\bQ(\mu_3)/\bQ}(\sigma)_K \cap \cs(K(\mu_3)/K)$ would at least be infinite and not more empty as for $K = \bQ$. The main obstruction now is that this set has Dirichlet density $0$, and no one of the usual arguments involving Dirichlet density will apply. To overcome this difficulty we will use the fixed Frobenius densities introduced above. Namely, it turns out that certain $x$-density of this set is positive and then one again can apply some density arguments. However, these arguments are in our situation much more subtle than in the situations where one can use Dirichlet density.

Finally, we remark that there are several other appraoches to realization results of similar spirit. As to the knowledge of the author, no one of them covers the abovementioned case, where one tries to realize $p$-extensions with ramification allowed only outside $\cs(K(\mu_p/K))$. We mention two recent approaches: a certain pro-$p$ version of the theorem above is also known (only for primes in $S$) in the much harder situation of a finite set $S$ by the work of A. Schmidt (cf. e.g. \cite{Sch}) but only after enlarging $S$ by an appropriate finite subset of a fixed set $T$ of primes of density $1$ (which, in particular, satisfies $(\dagger)_p$). A further, completely different and very powerful approach using automorphic forms, which deals with the whole pro-finite group and a finite set $S$, was introduced by Chenevier and Clozel \cite{Ch}, \cite{CC}. However, compared to results of this paper, the drawback is that one has to forget about solvability conditions and to assume $R = \emptyset$ (no control of 
the unramified extensions) and that at least one rational prime must lie in $\caO_{K,S}^{\ast}$.

\subsection*{Notation}

For any $a \in \bR$ we denote by $\bH_a$ the complex right half plane $\{s + it \colon \Re(s) > a \}$.
Let $G$ be a group and $\sigma \in G$ be any element. Then we denote by $C(\sigma, G)$ the conjugacy class of $\sigma$, by $\ord(\sigma)$ the order of $\sigma$ and by $Z_G(\sigma)$ the centralizer of $\sigma$. 

Let $L/K$ be an extension of number fields. We write $\Sigma_K$ for the set of all primes of $K$, $S_{\fp}(L)$ for the set of primes in $L$ lying over a prime $\fp$ of $K$. If $S \subseteq \Sigma_K$, then we write $S_L$, $S(L)$ or sometimes simply $S$ for the pull-back of $S$ to $L$. If $L/K$ is Galois and $x \in \Gal_{L/K}$, the Chebotarev set $P_{L/K}(x)$ is the set of primes in $K$ which are unramified in $L/K$ and whose Frobenius class is $C(x,\Gal_{L/K})$ and $P_{L/K}^x$ denotes the set of primes in $L$ which are unramified in $L/K$ and whose Frobenius is $x$. Moreover, we call a set which differs from a Chebotarev set only by a subset of Dirichlet density $0$ an almost Chebotarev set. For $\fp \in \Sigma_K$, $\N\fp$ denotes the norm of $\fp$ over $\bQ$, i.e., the cardinality of the residue field. If $S,T \subseteq \Sigma_K$, then $S \subsetsim T$ means that $S$ lies in $T$ up to a (Dirichlet) density zero subset and $S \backsimeq T$ means $S \subsetsim T$ and $T \subsetsim S$.
% \item $\cs(L/K) := P_{L/K}(1)$. 

% If $S$ is a set of primes of $K$, we write $K_S^R$ for the maximal solvable extension of $K$ unramified outside $S$ and completely split in $R$. Further, let $\fp$ be some prime of $K$ and choose some fixed extension $\bap$ of $\fp$ to $K_S^R$. We write $K_{S,\fp}^R$ for the union of all completions $L_{\bap|_L}$, taken over all finite subextensions $K_S^R/L/K$. Then the field $K_{S,\fp}^R$ is henselian and its completion is equal to the $\bap$-adic completion of $K_S^R$.

%************************************************************************************************************************************************************
%************************************************************************************************************************************************************

\subsection*{Outline of the paper}
In Section \ref{sec:Def_of_gen_densities} we define the generalized densities. In Section \ref{sec:pull-back-properties} we establish some base-change formulas and an easy generalization of Chebotarev's density theorem for these densities. In Section \ref{sec:gen_of_gen_dens_to_char_dens} we generalize slightly the notion of these densities introduced in Section \ref{sec:Def_of_gen_densities}. In Section \ref{sec:proof_of_thm} we prove Theorem \ref{thm:real_of_loc_ext}.

\subsection*{Acknowledgements} 

I want to thank Jakob Stix, Kay Wingberg, Jochen G\"artner and Daniel Harrer for fruitful discussions on this subject, from which I learned a lot. Also I want to thank the hospitality of the HIM center at the University of Bonn, where parts of this work were done and the Technical Univeristy Munich, where the rest was done.

%************************************************************************************************************************************************************
%************************************************************************************************************************************************************

%************************************************************************************************************************************************************
%************************************************************************************************************************************************************

%************************************************************************************************************************************************************
%************************************************************************************************************************************************************

%************************************************************************************************************************************************************
%************************************************************************************************************************************************************

\section{Densities associated to Frobenius elements} \label{sec:Def_of_gen_densities}

Let $K_0$ be a fixed finite extension of $\bQ$. Let $K/K_0$ be a finite Galois extension and $x \in G := \Gal_{K/K_0}$ an element of order $d$. Our starting point is the following easy but fundamental observation.

\begin{lm} \label{lm:pipipip} 
The series $\sum_{\fp \in P_{K/K_0}^x} \N\fp^{-s}$ converges for all $s$ with $\Re(s) > d^{-1}$. It defines a holomorphic function on $\bH_{d^{-1}}$ and 

\[ \lim_{s \rar d^{-1} + 0} \sum_{\fp \in P_{K/K_0}^x} \N\fp^{-s} = \infty. \]
\end{lm}

\begin{proof} For $\fp \in P_{K/K_0}^x$ with $\fp_0 := \fp|_{K_0}$ we have $\N\fp = \N\fp_0^d$. The map $P_{K/K_0}^x \rar P_{K/K_0}(x)$ is surjective and $\abs{\frac{Z_G(x)}{\langle x \rangle}}$-to-1 (this is immediate; cf. also Lemma \ref{lm:Faser_gamma}). Hence for all $s \in \bH_{d^{-1}}$ and all $a > 0$ we get:

\begin{equation}\label{eq:sumsumsum} 
\left(\abs{\frac{Z_G(x)}{\langle x \rangle}}\right)^{-1} \sum_{\substack{\fp \in P_{K/K_0}^x \\ \N\fp < a^d}} |\N\fp^{-s}| = \sum_{\substack{\fp_0 \in P_{K/K_0}(x) \\ \N\fp_0 < a}} |\N\fp_0^{-ds}| \leq \sum_{\substack{\fp_0 \in \Sigma_{K_0} \\ \N\fp_0 < a}} |\N\fp_0^{-ds}|. 
\end{equation}
 
\noindent The last term converges for $a \rar \infty$ and any fixed $s \in \bH_{d^{-1}}$. One sees easily that the convergence is uniform on the half plane $\bH_{d^{-1} + \epsilon}$ for any $\epsilon > 0$, hence the series in the lemma defines a holomorphic function on $\bH_{d^{-1}}$. Finally, $\sum_{\fp \in \Sigma_{K_0}} \N\fp^{-s}$ goes to infinity if $s \rar 1$ and $0 < \delta_{K_0}(P_{K/K_0}(x)) = \lim\limits_{s \rar 1}  \frac{\sum_{\fp \in P_{K/K_0}(x)} \N\fp^{-s}}{\sum_{\fp \in \Sigma_{K_0}} \N\fp^{-s}}$, hence also $\sum_{\fp \in P_{K/K_0}(x)} \N\fp^{-s} \rar \infty$ for $s \rar 1$, and the last statement of the lemma follows from \eqref{eq:sumsumsum}.
\end{proof}

\begin{Def} Let $K/K_0$ be a finite Galois extension and $S \subseteq \Sigma_K$ a set of primes of $K$. 
For $x \in \Gal_{K/K_0}$ we call the real number 
\[ \delta_{K/K_0,x}(S) := \lim_{s \rar \ord(x)^{-1} + 0} \frac{\sum_{\fp \in S \cap P_{K/K_0}^x} \N\fp^{-s} }{\sum_{\fp \in P_{K/K_0}^x} \N\fp^{-s} }, \] 
\noindent if it exists, the density of $S$ with respect to $x$ (over $K_0$), or simply, the $x$-density of $S$.
\end{Def}

\bigskip
\bigskip

\begin{rems}\label{rem:x-density} \mbox{}
\begin{itemize}
\item[(i)] The $x$-density satisfies the usual properties: If exists, $\delta_{K/K_0,x}(S)$ is a real number lying in the interval $[0,1]$. If $\delta_{K/K_0,x}(S) = 0$, then for any $S^{\prime} \subseteq S$, the $x$-density $\delta_{K/K_0,x}(S^{\prime})$ also exists and is $0$. By Lemma \ref{lm:pipipip}, finite sets of primes are irrelevant for the $x$-density: if $S$ and $T$ differ only by a finite set of primes, then $\delta_{K/K_0,x}(S)$ exists if and only $\delta_{K/K_0,x}(T)$ exists and if this is the case, then they are equal. Let $S, T$ be two sets of primes of $K$ having an $x$-density. If $S \cap T$ or $S \cup T$ has an $x$-density, then the second set does too and 
\[\delta_{K/K_0,x}(S) + \delta_{K/K_0,x}(T) = \delta_{K/K_0,x}(S \cap T)  + \delta_{K/K_0,x}(S \cup T). \]
\item[(ii)] More interesting, $\delta_{K/K_0,x}$ is essentially independent of $K_0$ as Lemma \ref{lm:dens_indep_of_K_0} below shows. Moreover, the density in $K$ with respect to a fixed Frobenius element over a smaller subfield can be defined simply over $\bQ$, but then (if $K/\bQ$ is not Galois and has Galois closure $K^n$) one has to deal with $\Gal_{K^n/K}$-cosets in $\Gal_{K^n/\bQ}$, instead of elements in a Galois group, which is definitely less nice. Thus we decide to stay by our approach.
\end{itemize}
\end{rems}

\begin{lm}\label{lm:dens_indep_of_K_0}
Assume $K/K_0^{\prime}/K_0$ are finite Galois extensions of $K_0$. Let $x \in \Gal_{K/K_0^{\prime}} \subseteq \Gal_{K/K_0}$ be of order $d$. Then for any set of primes $S$ in $K$ we have: $\delta_{K/K_0^{\prime},x}(S)$ exists if and only $\delta_{K/K_0,x}(S)$ exists and if this is the case, then they are equal. 
\end{lm}

\begin{proof}
Indeed, the sum $\sum_{P_{K/K_0^{\prime}}^x \sm \cs(K_0^{\prime}/K_0)(K)} \N\fp^{-s}$ is bounded for $s \rar d^{-1} + 0$ (since the inertia degree over $K_0$ and hence also over $\bQ$ of primes in this set is bigger than $d$) and $P_{K/K_0^{\prime}}^x \cap \cs(K_0^{\prime}/K_0)(K) = P_{K/K_0}^x$ and hence by Lemma \ref{lm:pipipip}:

\[ \delta_{K/K_0^{\prime},x}(S) = \lim_{s \rar d^{-1} + 0} \frac{\sum_{\fp \in S \cap P_{K/K_0^{\prime}}^x} \N\fp^{-s} }{\sum_{\fp \in P_{K/K_0^{\prime}}^x} \N\fp^{-s} } = \lim_{s \rar d^{-1} + 0} \frac{\sum_{\fp \in S \cap P_{K/K_0}^x} \N\fp^{-s} }{\sum_{\fp \in P_{K/K_0}^x} \N\fp^{-s} } = \delta_{K/K_0,x}(S). \]
\noindent (when both exist). 
\end{proof}

%************************************************************************************************************************************************************
%************************************************************************************************************************************************************

\section{Pull-back properties of $\delta_{K/K_0,x}$} \label{sec:pull-back-properties} 

We fix the following setting in this section: $L/K/K_0$ are finite Galois extensions, $G := \Gal_{L/K_0}$, $H := \Gal_{L/K}$, $\pi \colon G \tar G/H$ the natural projection, $x \in G/H$. For $y \in \pi^{-1}(x)$, let $\pr = \pr_{L/K} \colon P_{L/K_0}^y \rar P_{K/K_0}^x$ denote the restriction of primes from $L$ to $K$.

\begin{lm} \label{lm:H-conj_rel}
Let $y,z \in \pi^{-1}(x)$. Then $\pr(P_{L/K_0}^y), \pr(P_{L/K_0}^z)$ are either disjoint or equal. They are equal if and only if $y, z$ are $H$-conjugate.
\end{lm}

\begin{proof}
Assume that $\pr(P_{L/K_0}^y) \cap \pr(P_{L/K_0}^z) \neq \emptyset$. Then there are primes $\fP \in P_{L/K_0}^y, \fQ \in P_{L/K_0}^z$ with $\fP|_K = \fQ|_K =: \fp$. Let $\fp_0 := \fp|_{K_0}$. The primes in $L$ lying over $\fp_0$ are in 1:1-correspondence with cosets of $\langle y \rangle = D_{\fP,L/K_0} \subseteq G$:
\[ G/ \langle y \rangle \stackrel{\sim}{\longrightarrow} S_{\fp_0}(L), \quad g \langle y \rangle \mapsto g\fP. \]
The Frobenius of $g\fP$ is $gyg^{-1}$; after reduction modulo $H$ we obtain the same correspondence for $K$: $(G/H)/\langle x \rangle = G/H\langle y \rangle \stackrel{\sim}{\rar} S_{\fp_0}(K)$ and $\fP,g\fP$ lie over the same prime of $K$ if and only if $\pi(g) \in \langle x \rangle$, i.e., $g \in H\langle y \rangle$. So with our assumption we get $\fQ = g\fP$ for some $g \in H \langle y \rangle$ with $gyg^{-1} = z$. By multiplying with a power of $y$, we can modify $g$ such that $g \in H$.

Assume conversely that for $y,z \in \pi^{-1}(x)$ there is some $g \in H$ with $gyg^{-1} = z$. Then we claim that $\pr(P_{L/K_0}^y) = \pr(P_{L/K_0}^z)$. Indeed, let $\fp \in \pr(P_{L/K_0}^y)$ with preimage $\fP \in P_{L/K_0}^y$. Using the above description of primes via cosets, it is immediate to see that $g\fP \in P_{L/K_0}^z$ also lies over $\fp$.
\end{proof}

For an $H$-conjugacy class $C$ in $\pi^{-1}(x)$, let $M_C \subseteq P_{K/K_0}^x$ denote the image of $P_{L/K_0}^y$ under $pr$ for some (any) $y \in C$. Thus if $\Ram(L/K)$ denotes the set of primes of $K$, which ramify in $L$, then we have a disjoint decomposition 
\[ P_{K/K_0}^x = (\Ram(L/K) \cap P^x_{K/K_0}) \cup \bigcup_{C \subseteq \pi^{-1}(x)} M_C, \] 
\noindent where the first set is finite and the union is taken over all $H$-conjugacy classes inside $\pi^{-1}(x)$. We have the following generalization of Chebotarev's density theorem (observe that $\abs{H} = \abs{\pi^{-1}(x)}$):

\begin{prop}\label{prop:verallgCheb}
Let $L/K/K_0,\pi,x$ be as above. Let $C$ be an $H$-conjugacy class in $\pi^{-1}(x)$. Then 
\[ \delta_{K/K_0,x}(M_C) = \frac{\abs{C}}{\abs{H}}. \]
\end{prop}

When setting $x = 1$, this reduces to the classical Chebotarev's density theorem for the Dirichlet density. Fortunately, the proof of this proposition does not need any new L-functions, it simply follows from the classical Chebotarev.

\begin{lm}\label{lm:Faser_gamma}
Let $L/K/K_0$, $\pi$, $x$ be as above. Let $d$ be the order of $x$ in $G/H$. Let $y \in \pi^{-1}(x)$ and let $C \subseteq \pi^{-1}(x)$ denote the $H$-conjugacy class of $y$. Then the map $\pr \colon P_{L/K_0}^y \tar M_C$ is surjective and $\gamma_{L/K}(y)$-to-$1$, where $\gamma_{L/K}(y) :=\frac{\abs{Z_H(y)}}{\abs{\langle y^d \rangle}}$.
\end{lm}

\begin{proof}[Proof of Lemma \ref{lm:Faser_gamma}]
The surjectivity follows from definition. Using the description of primes via cosets modulo the decomposition group, one sees easily that for $\fp \in M_C$, the primes in $S_{\fp}(L) \cap P_{L/K_0}^y$ are in one-to-one correspondence with elements in the group $Z_G(y) \cap H\langle y \rangle / \langle y \rangle$. One sees then that the composition 
\[Z_H(y) \har Z_G(y) \cap H \langle y \rangle \tar Z_G(y) \cap H \langle y \rangle / \langle y \rangle \]
\noindent is surjective and its kernel is $\langle y^d \rangle$.
\end{proof}

\begin{proof}[Proof of Proposition \ref{prop:verallgCheb}]
By preceding lemmas, we have the following diagram:

\centerline{
\begin{xy}
\xymatrix{
& P_{L/K_0}^y \ar@{->>}[d]_{\gamma_{L/K}(y)} \ar@/^2pc/[dd]^{\gamma_{L/K_0}(y)} \\
P_{K/K_0}^x \ar@{->>}[d]_{\gamma_{K/K_0}(x)} & M_C \ar@{_{(}->}[l] \ar@{->>}[d] \\ 
P_{K/K_0}(x) & \ar@{_{(}->}[l] P_{L/K_0}(y)} 
\end{xy}
}

\noindent in which any vertical map is surjective and has fibers of equal cardinality, and the number on the arrow denotes the dergee ($\gamma$ is as in Lemma \ref{lm:Faser_gamma}).  Thus the lower right map is $\beta(y):1$, with $\beta(y) = \frac{\abs{Z_G(y)}}{\abs{\langle x \rangle} \abs{Z_H(y)}}$. It follows:

\begin{equation} \nonumber
\delta_{K/K_0,x}(M_C) = \lim_{s \rar d^{-1}+0} \frac{\sum_{M_C} \N\fp^{-s}}{\sum_{P_{K/K_0}^x} \N\fp^{-s}} = \lim_{t \rar 1+0} \frac{\beta(y)\sum_{\fp \in P_{L/K_0}(y)} \N\fp^{-t} }{\gamma_{K/K_0}(x)\sum_{\fp \in P_{K/K_0}(x)} \N\fp^{-t} } = \frac{\beta(y) \delta_{K_0}(P_{L/K_0}(y)) }{\gamma_{K/K_0}(x) \delta_{K_0}(P_{K/K_0}(x))}.
\end{equation}

\noindent where $\delta_{K_0}$ denotes the usual Dirichlet density on $\Sigma_{K_0}$. By Chebotarev we have: $\delta_{K_0} (P_{L/K_0}(y)) = \frac{\abs{C(y,G)}}{\abs{G}} = \frac{1}{\abs{Z_{G}(y)}}$ and $\delta_{K_0}(P_{K/K_0}(x)) = \frac{1}{\abs{Z_{G/H}(x)}}$. Hence we obtain:

\[ \delta_{K/K_0,x}(M_C) = \frac{\beta(y) \abs{Z_{G/H}(x)}}{\gamma_{K/K_0}(x) \abs{Z_G(y)}} = \frac{\abs{Z_G(y)}}{\abs{\langle x \rangle}\abs{Z_H(y)}} \frac{\abs{\langle x \rangle}}{\abs{Z_{G/H}(x)}} \frac{\abs{Z_{G/H}(x)}}{\abs{Z_G(y)}} = \frac{1}{\abs{Z_H(y)}} = \frac{\abs{C}}{\abs{H}}. \qedhere \]

\end{proof}

Now we can derive the pull-back behavior of $\delta_{K/K_0,x}$.

\begin{cor}\label{cor:bc_of_delta_x}
Let $y \in \pi^{-1}(x)$ and let $C$ be its $H$-conjugacy class in $\pi^{-1}(x)$. Then 
\[ \delta_{L/K_0,y}(S_L) = \delta_{K/K_0,x}(M_C)^{-1}\delta_{K/K_0,x}(S \cap M_C) = \frac{\abs{H}}{\abs{C}} \delta_{K/K_0,x}(S \cap M_C) \]
\noindent if all densities exist.
\end{cor}

\begin{proof}
Let $e$ denote the order of $y$ in $G$ and $d$ the order of $x$ in $G/H$. Then 
\begin{eqnarray*} 
\delta_{L/K_0,y}(S_L) &=& \lim_{s \rar e^{-1}+0} \frac{\sum_{\fp \in S_L \cap P_{L/K_0}^y} \N\fp^{-s} }{\sum_{\fp \in P_{L/K_0}^y} \N\fp^{-s}} = \lim_{t \rar d^{-1}+0} \frac{\sum_{\fp \in S \cap M_C} \N\fp^{-t} }{\sum_{\fp \in M_C} \N\fp^{-t}} \\
&=& \delta_{K/K_0,x}(M_C)^{-1}\delta_{K/K_0,x}(S \cap M_C), 
\end{eqnarray*}

\noindent where we made a change of variables by replacing $s$ by $t := \frac{e}{d} s$ and used the fact that $S_L$ is defined over $K$. Proposition \ref{prop:verallgCheb} finishes the proof.
\end{proof}

The special case $x = y = 1$ in Corollary \ref{cor:bc_of_delta_x} gives the well-known formula 
\[ \delta_L(S_L) = [L:K]\delta_K(S \cap \cs(K/K_0)).\] 
\noindent for the Dirichlet density. We compute the $x$-density of pull-backs of Chebotarev sets.

\begin{cor}\label{cor:counting_inf_dens_of_Chebotarevs}
Let $L,M$ be two finite Galois extensions of $K$. Let $\sigma \in \Gal_{M/K}$, $x \in \Gal_{L/K}$ with images $\bar{\sigma}, \bar{x}$ in $\Gal_{L \cap M/K}$ respectively. Let $S \backsimeq P_{M/K}(\sigma)$. Then 
\[ \delta_{L/K,x}(S_L) = \begin{cases} \frac{ \abs{C((x,\sigma), \Gal_{LM/K})} }{ [M:L \cap M]\abs{C(x, \Gal_{L/K})} } &  \text{ if } \bar{\sigma} = \bar{x}, \\ 0 & \text{ if } \bar{\sigma} \neq \bar{x}, \end{cases} \]
where we write $(x,\sigma)$ for the unique element of $\Gal_{LM/K} \cong \Gal_{L/K} \times_{\Gal_{L \cap M/K}} \Gal_{M/K}$ mapping to $x,\sigma$ under both projections.
\end{cor}

\begin{proof}
Indeed, apply Corollary \ref{cor:bc_of_delta_x} to $\delta_{K/K,1}$ and $\delta_{L/K,x}$. Then $\delta_{K/K,1} = \delta_K$ is the Dirichlet density and we have $M_x = P_{L/K}(x)$ and 
\[ \delta_{L/K,x}(S_L) = \delta_K(P_{L/K}(x))^{-1} \delta_K(P_{L/K}(x) \cap S) = \delta_K(P_{L/K}(x))^{-1} \delta_K(P_{L/K}(x) \cap P_{M/K}(\sigma)). \]
\noindent The intersection $P_{L/K}(x) \cap P_{M/K}(\sigma)$ is empty unless $\bar{\sigma} = \bar{x}$, hence we can assume equality. Under this assumption, we have $P_{L/K}(x) \cap P_{M/K}(\sigma) = P_{LM/K}((\sigma,x))$ and the corollary follows immediately from Chebotarev.
\end{proof}

\begin{comment}
\begin{rem}
The drawback of this proposition is that in general the size of $C((x,\sigma); \Gal_{LM/K})$ can not simply be describen by $\sigma,x$. Things simplify if one is in the special case $x = 1$ (which is \cite{Iv} Prop. ...). This illustrates the general fact that Dirichlet density is simplier to handle as the $x$-density in general, because of the group theoretic subtleties, coming from the Galois groups.
\end{rem}
\end{comment}

%************************************************************************************************************************************************************
%************************************************************************************************************************************************************

\section{Densities associated to characters} \label{sec:gen_of_gen_dens_to_char_dens} 

\begin{Def}
Let $K/K_0$ be finite Galois and $S \subseteq \Sigma_K$ a subset.
\begin{itemize}
\item[(i)]  We call $S$ \emph{mesurable} (over $K_0$), if for all $x \in \Gal_{K/K_0}$ the density $\delta_{K/K_0,x}(S)$ exists (this is essentially independent of $K_0$).
\item[(ii)] Assume that $S$ is mesurable. Then define the \emph{characteristic function} of $S$ as 
\[ \chi_{K/K_0,S} \colon \Gal_{K/K_0} \rar [0,1], \quad x \mapsto \delta_{K/K_0,x}(S). \]
\end{itemize}
\end{Def}

% \item[(i)] It is not obvious how to deal with non-Galois extensions, so we just omit the problem, by going up to a Galois extension. For simplicity we choose the minimal one, i.e., the Galois closure of $K/\bQ$. We do not think that it makes an essential difference to allow $K^n$ to be any finite Galois extension of $\bQ$ which contains $K$. 

\begin{rem}
Notice that $\chi_{K/K_0,S}$ is only a real-valued function on $\Gal_{K/K_0}$, which is not necessarily a class function. But, if $S$ is defined over $K_0$, it is a class function (clearly, the converse is in general not true).
\end{rem}

For a finite group $G$, let $G(\bC)$ be the set of complex valued functions on $G$. Then we have the inner product on $G(\bC)$ defined by 
\[ \langle \chi,\psi \rangle_G := \abs{G}^{-1} \sum_{x \in G} f(x)\overline{g(x)} \]

\noindent for all $\chi, \psi \in G(\bC)$. If $G = \Gal_{L/K}$ we also write $\langle \cdot , \cdot \rangle_{L/K}$ (or even $\langle \cdot , \cdot \rangle_L$ if $K$ is clear from the context) instead of $\langle \cdot , \cdot \rangle_{\Gal_{L/K}}$.

\begin{Def} \label{def:psi-mes}
Let $K/K_0$ be finite Galois. For any $\psi \in \Gal_{K/K_0}(\bC)$ we define the $\bC$-valued function $\delta_{K/K_0,\psi}$ on the set of all mesurable subsets of $\Sigma_K$ by 
\[ \delta_{K/K_0,\psi} (S) :=  \langle \psi, \chi_{K/K_0,S} \rangle_K \]
for any mesurable set $S$. We say that $\delta_{K/K_0,\psi}$ is a \emph{density}, if for all mesurable $S$ it takes values in the real unit interval $[0,1]$ and $\delta_{K/K_0,\psi}(\Sigma_K) = 1$.
\end{Def}

\begin{lm}
Let $K/K_0$ be finite Galois and $\psi \in G_{K/K_0}(\bC)$. 

\begin{itemize}
\item[(i)]  Let $S, T$ be mesurable. If one of the sets $S \cap T, S \cup T$ is mesurable, then the second set is too and
\[\delta_{K/K_0,\psi}(S) + \delta_{K/K_0,\psi}(T) = \delta_{K/K_0,\psi}(S \cap T)  + \delta_{K/K_0,\psi}(S \cup T). \]
\item[(ii)] The function $\delta_{K/K_0,\psi}$ is a density if and only if $\psi$ takes values only in the real interval $[0,[K:K_0]]$ and $\langle \psi, {\bf 1} \rangle_{K/K_0} = 1$, where ${\bf 1}$ denotes the trivial character of $\Gal_{K/K_0}$.
\end{itemize}
\end{lm}

\begin{proof} (i) follows from bilinearity of $\langle \cdot,\cdot\rangle_{K/K_0}$ and Remark \ref{rem:x-density} and (ii) is an immediate computation. \end{proof}

\begin{rem}
In particular, the Dirichlet density $\delta_K$ corresponds to the character of the regular representation of $\Gal_{K/K_0}$ and $\delta_{K/K_0,x}$ for $x \in \Gal_{K/K_0}$ corresponds to the function defined by $\psi(y) = [K:K_0]\delta_{xy}$, where $\delta_{xy}$ is the Kronecker symbol.
\end{rem}

The next proposition shows that if $L/K$ is a finite extension then $\chi_{L,S_L}$ is in a sense the induction of $\chi_{K,S}$ to $L$:

\begin{prop} 
Let $L/K/K_0$ be finite Galois extensions. Denote by $\pi : \Gal_{L/K_0} \tar \Gal_{K/K_0}$ the natural projection. Then for all $\psi \in \Gal_{K/K_0}(\bC)$ and all mesurable $S$ we have
\[ \langle \inf\nolimits_{\Gal_{K/K_0}}^{\Gal_{L/K_0}} \psi, \chi_{L,S_L} \rangle_L = \langle \psi, \chi_{K,S} \rangle_K \]
\noindent or equivalently,
\[\delta_{L/K_0, \psi \circ \pi } (S_L) = \delta_{K/K_0,\psi}(S). \]
\end{prop}
\begin{proof}
Let $G := \Gal_{L/K_0}$, $H := \Gal_{L/K}$. For $y \in G$, let $C(y) \subseteq \pi^{-1}(\pi(y))$ denote its $H$-conjugacy class. Then (we write $\delta_{\ast,\psi}$ instead of $\delta_{\ast/K_0,\psi}$):
\begin{eqnarray*}
\langle \psi \circ \pi, \chi_{L,S_L} \rangle_L &=& \frac{1}{\abs{G}} \sum_{y \in G} \psi(\pi(y))\delta_{L,y}(S_L) \\ 
&=& \frac{1}{\abs{G}} \sum_{y \in G} \psi(\pi(y)) \frac{\abs{H}}{\abs{C(y)}} \delta_{K,x}(S \cap M_{C(y)}) \\ 
&=& \frac{1}{\abs{(G/H)}} \sum_{x \in G/H} \psi(x) \sum_{C \subseteq \pi^{-1}(x)} \sum_{y \in C} \frac{1}{\abs{C}} \delta_{K,x}(S \cap M_C) \\
&=& \frac{1}{\abs{(G/H)}} \sum_{x \in G/H} \psi(x) \sum_{C \subseteq \pi^{-1}(x)} \delta_{K,x}(S \cap M_C) \\
&=& \frac{1}{\abs{(G/H)}} \sum_{x \in G/H} \psi(x) \delta_{K,x}(S) = \langle \psi, \chi_{K,S} \rangle_K. 
\end{eqnarray*}
\noindent where the second equality follows from Corollary \ref{cor:bc_of_delta_x}.
\end{proof}

%************************************************************************************************************************************************************
%************************************************************************************************************************************************************

\section{Realization of local extensions} \label{sec:proof_of_thm}

\subsection{Complements on stable sets}\label{sec:complem_stable_sets}

Before starting with the proof of Theorem \ref{thm:real_of_loc_ext}, we recall for the convenience of the reader some definitions and results from \cite{IvStableSets}.

\begin{Def}[part of \cite{IvStableSets} Definitions 2.4, 2.7] Let $S$ be a set of primes of $K$ and $\cL/K$ any (algebraic) extension. 
\begin{itemize}
\item[(i)] Let $\lambda > 1$. A finite subextension $\cL/L_0/K$ is $\lambda$-stabilizing for $S$ for $\cL/K$, if there exists a subset $S_0 \subseteq S$ and some $a \in (0,1]$, such that $\lambda a > \delta_L(S_0) \geq a > 0$ for all finite subextensions $\cL/L/L_0$. We say that $S$ is \textbf{$\lambda$-stable}, if it has a $\lambda$-stabilizing extension for $\cL/K$. We say that $S$ is \textbf{stable for $\cL/K$}, if it is $\lambda$-stable for $\cL/K$ for some $\lambda > 1$. We say that $S$ is ($\lambda$-)\textbf{stable}, if it is ($\lambda$-)stable for $K_S/K$. 

\item[(ii)] We say that $S$ is \textbf{persistent} for $\cL/K$ (with persisting field $L_0$, lying between $\cL/K$) if the density of a subset $S_0 \subseteq S$ gets constant in the tower $\cL/L_0$.

\item[(iii)] Let $p$ be a rational prime. We say that $(S, \cL/K)$ satisfies $(\dagger)_p^{\rm rel}$, if $\mu_p \subseteq \cL$ and $S$ is $p$-stable for $\cL/K$ or $\mu_p \not\subseteq \cL$ and $S$ is stable for $\cL(\mu_p)/K$. We say that $S$ satisfies $(\dagger)_p$, if $(S, K_S/K)$ satisfies $(\dagger)_p^{\rm rel}$.
\end{itemize}  
\end{Def}

We will need the following crucial result about stable sets, which we take from \cite{IvStableSets}. 

\begin{thm}[\cite{IvStableSets} Theorem 5.9] \label{thm:dirlim_GW_coker_vanishes_citloc} Let $K$ be a number field, $S$ a set of primes of $K$ and $\cL \subseteq K_S$ a subextension normal over $K$, such that $(S, \cL)$ satisfies $(\dagger)_p^{\rm rel}$. Let $T$ be a finite set of primes of $K$ containing $(S_p \cup S_{\infty}) \sm S$. If $p^{\infty} | [\cL:K]$, then
\[ \dirlim_{\cL/L/K, \res} \coker^1(K_{S \cup T}/L, T, \bZ/p\bZ) = 0. \]
\end{thm}

\begin{rem} Many results (e.g., such as the one quoted above, but also finite cohomological dimension, etc.) holding for sets with Dirichlet density one also hold (with respect to a prime $p$) for stable sets of primes (satisfying $(\dagger)_p$). The proofs in the case of sets with density one rely heavily on the fact that various Tate-Shafarevich groups of $\Gal_{K,S}$ with finite resp. divisible coefficients vanish. This is in general not true for stable sets and the reason why many proofs (in particular, the proof of Theorem \ref{thm:dirlim_GW_coker_vanishes_citloc}) still work, is that one can, using stability conditions, bound the size of Tate-Shafarevich groups, which in turn implies the vanishing of them in the limit taken over all finite subextensions of certain (infinite) subextensions $K_S/\cL/K$. 
\end{rem}

By easy density computations we obtain:

\begin{lm}[\cite{IvStableSets} Proposition 3.3, Corollary 3.4]\label{lm:ACS_persistent}
Let $M/K$ be a finite Galois extension and $\sigma \in \Gal_{M/K}$.

\begin{itemize}
\item[(i)] Let $L/K$ be any finite extension. Let $L_0 := L \cap M$. Then:
\[
\delta_L(P_{M/K}(\sigma)_L) = \frac{ \abs{ C(\sigma; \Gal_{M/K}) \cap \Gal_{M/L_0} } }{ \abs{\Gal_{M/L_0}} }.
\]
\item[(ii)] Let $S \backsimeq P_{M/K}(\sigma)$. Let $\cL/K$ be any extension. Then $S$ is persistent for $\cL/K$ with persisting field $L_0$ if and only if
\[ \Gal_{M/M \cap \cL} \cap C(\sigma; \Gal_{M/K}) \neq \emptyset, \]
where $C(\sigma;\Gal_{M/K})$ denotes the conjugacy class of $\sigma$ in $\Gal_{M/K}$. 
\end{itemize}
\end{lm}

From now on and until the end of the paper we prove Theorem \ref{thm:real_of_loc_ext}. We let $M/K,\sigma,R \subseteq S$ and $\ell \leq \infty$ be as in the theorem.

%************************************************************************************************************************************************************
%************************************************************************************************************************************************************

\subsection{Some reduction steps} \label{sec:red_steps}

Clearly, we can assume $\ell < \infty$. For any finite subextension $K_S^R/L/K$, any finite set $T$ of primes of $L$ and any rational prime $p$ consider the cokernel 

\[ \coh^1(K_{S \cup T}/L, \bZ/p\bZ) \rar \prod_T \coh^1(\overline{K_{\fp}}/L_{\fp}, \bZ/p\bZ) \tar \coker^1(K_{S \cup T}/L,T;\bZ/p\bZ) \]

\noindent of the restriction map. Theorem \ref{thm:real_of_loc_ext} for $K_S^R(\fc_{\leq \ell})/K$ follows easily from Claim \ref{claim:claim1} below for all $p \leq \ell$ (cf. \cite{NSW} 9.2.7, 9.4.3).

\begin{claim}\label{claim:claim1} For all $T \supseteq R \cup S_p \cup S_{\infty}$, we have

\[ \dirlim_L \coker^1(K_{S \cup T}/L, T; \bZ/p\bZ) = 0, \]

\noindent where the limit is taken over all finite subextensions $L$ of $K_S^R(\fc_{\leq \ell})/K$.
% For all finite subextensions $K_S^R(\fc_{\leq p})/L/K$ and all triples $(\ell, T, \alpha = (\alpha_{\fp})_{\fp \in T})$ consisting of a rational prime $\ell \leq p$, a finite set $T \supseteq R_L \cup S_{\ell,L} \cup S_{\infty,L}$ of primes of $L$ and a collection of classes $\alpha_{\fp} \in \coh^1(L_{\fp},\bZ/\ell\bZ)$ for $\fp \in T$, such that $\alpha_{\fp}$ is $0$ if $\fp \in R_L$ and unramified if $\fp \not\in S_L$, there is a finite subextension $K_S^R(\fc_{\leq p})/L^{\prime}/L$ such that the image of $\alpha$ in $\coker^1(K_{S \cup T}/L^{\prime},T;\bZ/\ell \bZ)$ vanishes.
\end{claim}

\begin{lm} \label{lm:red_to_persistent}
There are two finite sets $R_1,R_2$ of primes of $K$ with $R_1 \cap R_2 = R$ and such that $M \cap K_S^{R_j} = K$, i.e., $P_{M/K}(\sigma)$ (and hence also $S$) is persistent for $K_S^{R_j}/K$ with persisting field $K$ for $i= 1,2$.
\end{lm}

\begin{proof}
Indeed, choose a set of generators $g_1,\dots, g_r$ of $\Gal_{M/K}$ and for $j=1,2$ primes $\fp_{j,1}, \dots, \fp_{j,r}$ of $K$ unramified in $M/K$ such that the Frobenius conjugacy class corresponding to $\fp_{j,k}$ is the conjugacy class of $g_k$ and such that the sets $\{ \fp_{j,k} \colon k = 1,\dots,r \} \sm R$ are disjoint for $j=1,2$ (this is possible by Chebotarev). Let $R_j := \{\fp_{j,k} \colon k = 1, \dots, r \} \cup R$. Then any non-trivial (Galois) subextension of $M/K$ is not completely split in at least one prime $\fp \in R_j$. Hence $M \cap K_S^{R_j} = K$ and hence by Lemma \ref{lm:ACS_persistent}, $P_{M/K}(\sigma)$ is persistent for $K_S^{R_j}/K$ with persisting field $K$. 
\end{proof}

\emph{Step 1.} By Lemma \ref{lm:red_to_persistent} we can enlarge $R$ and hence assume that $M$ satisfies $M \cap K_S^R = K$. In particular, $P_{M/K}(\sigma)$ is  persistent for $K_S^R/K$ with persisting field $K$ by Lemma \ref{lm:ACS_persistent} (note also that the assumptions of the theorem are inherited if we replace $K$ by a finite subextension $K_S^R(\fc_{\leq \ell})/L/K$ and $S$ by $S_L$, as $P_{ML/L}(\sigma) \backsimeq P_{M/K}(\sigma)_L$ for any such $L$ and since we also have $L_S^R(\fc_{\leq \ell}) = K_S^R(\fc_{\leq \ell})$ and $ML \cap L_S^R(\fc_{\leq \ell}) = L$ (as $K_S^R(\fc_{\leq \ell}) \cap M = K$)). Now Claim \ref{claim:claim1} for all $p \leq \ell$ such that $(S,K_R^S(\fc_{\leq \ell})/K)$ satisfies $(\dagger)_p^{\rm rel}$, follows by Theorem \ref{thm:dirlim_GW_coker_vanishes_citloc} (observe, in particular, that since $P_{M/K}(\sigma)$ is persistent for $K_S^R/K$ and $\mu_2 \subseteq K$, $(\dagger)_2^{\rm rel}$ is always satisfied).

\emph{Step 2.} Thus we can assume that $(P_{M/K}(\sigma),K_R^S(\fc_{\leq \ell})/K)$ does not satisfy $(\dagger)_p^{\rm rel}$. By induction we assume that Claim \ref{claim:claim1} holds for all $p^{\prime} < p$. As the assumptions are stable under enlarging $K$ inside $K_S^R(\fc_{\leq \ell})$, it is enough to show that for each $T$ as in the claim, there is a (not necessarily finite) subextension $K_S^R(\fc_{\leq \ell})/\cL/K$, such that 
\begin{equation}\label{eq:claimseq}
\dirlim_{\cL/L/K} \coker^1(K_{S \cup T}/L, T; \bZ/p\bZ) = 0,
\end{equation}

\noindent where the limit is taken over finite subextensions of $\cL/K$. Further, since $P_{M/K}(\sigma)$ is persistent for $K_S^R(\fc_{\leq \ell})/K$, our assumption implies $\mu_p \not\subseteq K_S^R(\fc_{\leq \ell})$ and $\delta_{L(\mu_p)}(P_{M/K}(\sigma)) = 0$ for $L$ a sufficently big finite subextension of $K_S^R(\fc_{\leq \ell})/K$.
% this follows from the fact that density of almost Chebotarev sets gets constant beginning with some finite subextension, so from non-stability follows that this constant is $0$). 
We replace $K$ by such $L$, and so we can assume that $\delta_{K(\mu_p)}(P_{M/K}(\sigma)) = 0$. 

\emph{Step 3.} We have $\mu_p \not\subseteq K_S^R(\fc_{\leq \ell})$ and after replacing $K$ by a finite subextension of $K_S^R(\fc_{\leq \ell})/K$ if necessary, we can assume that for any finite subextension $K_S^R(\fc_{\leq \ell})/L/K$, the natural map $\Gal_{L(\mu_p)/L} \rar \Gal_{K(\mu_p)/K}$ is an isomorphism. We write $\Delta := \Gal_{K(\mu_p)/K}$ and $d := \ord(\Delta)$. The group $\Delta$ can canonically be identified with a subgroup of $\bF_p^{\ast}$, an element $x \in \bF_p^{\ast}$ acting on $\zeta \in \mu_p$ by $\zeta \mapsto \zeta^x$. Note that by assumption we have $1 < d < p$.

\emph{Step 4.} We replace $M$ by $M(\mu_p)$. Therefore consider the following diagram of extensions of $K$:

\centerline{
\begin{xy}
\xymatrix{
 & M(\mu_p) \ar@{-}[ld] \ar@{-}[rd] \ar@{-}[rdd] & &  \\
M \ar@{-}[rd] & & K(\mu_p) \ar@{-}[ld] &  \\
& M \cap K(\mu_p) \ar@{-}[d] & K^{\prime} = K_S^R(\fc_{\leq \ell}) \cap M(\mu_p) \ar@{-}[ld] & \\
& K & & & \\
}
\end{xy}
} 

\noindent We have $\Gal_{M(\mu_p)/K} = \Gal_{M/K} \times_{\Gal_{M \cap K(\mu_p)/K}} \Gal_{K(\mu_p)/K}$. Let $K^{\prime} := K_S^R \cap M(\mu_p)$. Then $K^{\prime} \cap M \subseteq K_S^R(\fc_{\leq \ell}) \cap M = K$, hence $\Gal_{M(\mu_p)/K^{\prime}}$ and $\Gal_{M(\mu_p)/M}$ together generate $\Gal_{M(\mu_p)/K}$ and hence the composition  $\Gal_{M(\mu_p)/K^{\prime}} \har \Gal_{M(\mu_p)/K} \tar \Gal_{M/K}$ is surjective. Let $\sigma^{\prime}$ be a preimage of $\sigma$ inside $\Gal_{M(\mu_p)/K^{\prime}} \subseteq \Gal_{M(\mu_p)/K}$. Then $P_{M(\mu_p)/K}(\sigma^{\prime}) \subseteq P_{M/K}(\sigma) \subsetsim S$ and $P_{M(\mu_p)/K^{\prime}}(\sigma^{\prime}) \backsimeq P_{M(\mu_p)/K^{\prime}}(\sigma^{\prime}) \cap \cs(K^{\prime}/K)_{K^{\prime}} = P_{M(\mu_p)/K}(\sigma^{\prime})_{K^{\prime}}$. Hence $P_{M(\mu_p)/K^{\prime}}(\sigma^{\prime}) \subsetsim S_{K^{\prime}}$. Thus we can replace $(K,P_{M/K}(\sigma))$ by $(K^{\prime},P_{M(\mu_p)/K^{\prime}}(\sigma^{\prime}))$ and, in particular, we can assume that $\mu_p \
subseteq M$. We have 
now the following easy situation:

\begin{equation}\label{eq:red_to_M_mu_p}
\begin{gathered}
\xymatrix{
MK_S^R(\fc_{\leq \ell}) \ar@{-}[r] & K_S^R(\fc_{\leq \ell})(\mu_p) \ar@{-}[r] & K_S^R(\fc_{\leq \ell}) \\
M \ar@{-}[u] \ar@{-}[r] & K(\mu_p) \ar@{-}[u] \ar@{-}[r] & K \ar@{-}[u] 
}
\end{gathered}
\end{equation}

\noindent and the right and the outer squares are cartesian, i.e., $M \cap K_S^R(\fc_{\leq \ell}) = K$ and $K(\mu_p) \cap K_S^R(\fc_{\leq \ell}) = K$. By Lemma \ref{lm:leftcart} also the left square is cartesian, i.e., $M \cap K_S^R(\fc_{\leq \ell})(\mu_p) = K(\mu_p)$. Observe also that the situation is now stable under replacing $K, K(\mu_p), M$ by $L, L(\mu_p),ML$ for a finite subextension $K_S^R(\fc_{\leq \ell})/L/K$ and the Galois groups $\Gal_{M/K}, \Gal_{M/K(\mu_p)}, \Gal_{K(\mu_p)/K} = \Delta$ will stay unchanged under such a replacement.

\begin{lm}\label{lm:leftcart}
In the above situation we have $M \cap K_S^R(\fc_{\leq \ell})(\mu_p) = K(\mu_p)$.
\end{lm}
\begin{proof}
We have natural homomorphisms $\Gal_{MK_S^R(\fc_{\leq \ell})/M} \rar \Gal_{K_S^R(\fc_{\leq \ell})(\mu_p)/K(\mu_p)} \rar \Gal_{K_S^R(\fc_{\leq \ell})/K}$. The right one and the composition of both are isomorphisms. Hence also the left one is an isomorphism.
\end{proof}

Observe that $C(\sigma, \Gal_{M/K}) \cap \Gal_{M/K(\mu_p)} = \emptyset$ since $\delta_{K(\mu_p)}(P_{M/K}(\sigma)) = 0$ (cf. Lemma \ref{lm:ACS_persistent}), and hence the image $\bar{\sigma}$ of $\sigma$ in $\Delta = \Gal_{K(\mu_p)/K}$ is unequal 1.

\emph{Step 5.} Let $\fp \not\in R$ be a prime of $K$. Recall the number $1 < d < p$ from step 3. By the induction assumption in step 2, we can realize a cyclic extension of order $d$ at $\fp$ by a finite subextension of $K_S^R(\fc_{\leq \ell})/K$. More precisely, there is a finite subextension $K_S^R(\fc_{\leq \ell})/K_0/K$ such that the decomposition group $D_{\fp_1,K_0/K}$ at a prime $\fp_1$ of $K_0$ lying over $\fp$ contains a cyclic subgroup $H_0$ of order $d$. We replace $K$ by $K_0^{H_0}$ (and $P_{M/K}(\sigma)$ by $P_{MK_0^{H_0}/K_0^{H_0}}(\sigma)$) and hence can assume that $K$ has a cyclic extension $K_0$ of degree $d$ inside $K_S^R(\fc_{\leq \ell})$.

% \item[(v)] Finally, we known that $M$ is such that $M \cap K_S^R = K$ and $\delta_{K(\mu_p)}P_{M/K}(\sigma) = 0$. Then by enlarging $M$ if necessary, we can assume that $M \supseteq K(\mu_p)$ and the image $\bar{\sigma}$ of $\sigma \in \Gal_{M/K}$ in $\Gal_{K(\mu_p)/K}$ is non-trivial.

We summarize the special situation obtained by all reduction steps: we have a number field $K$, two sets of primes $S \supseteq R$ of $K$ with $R$ finite. We have further a finite extension $M/K(\mu_p)/K$ such that all squares in the diagram \eqref{eq:red_to_M_mu_p} in step 4 are cartesian, an element $\sigma \in \Gal_{M/K}$ with $P_{M/K}(\sigma) \subsetsim S$ and image $1 \neq \bar{\sigma} \in \Delta = \Gal_{K(\mu_p)/K}$. We have $d := \abs{\Delta}$ with $1 < d < p$, and there is a finite cyclic subextension $K_S^R(\fc_{\leq \ell})/K_0/K$ of degree $d$ with Galois group $H_0 := \Gal_{K_0/K}$. In this very special situation we want to show Claim \ref{claim:claim1} for $p$. As remarked in step 2, it is enough to show that for each finite set $T \supseteq R \cup S_p \cup S_{\infty}$, there is a subextension $K_S^R(\fc_{\leq \ell})/\cL/K$, such that $\eqref{eq:claimseq}$ holds. Recall that Poitou-Tate duality implies a surjection:

\[ \Sha^1(K_{S \cup T}/K, S \sm T, \mu_p)^{\vee} \tar \coker^1(K_{S \cup T}/K, T; \bZ/p\bZ) \]

\noindent (cf. \cite{NSW} 9.2.2), where the transition maps $\res$ on the right correspond to $\cores^{\vee}$ on the left. By exactness of $\dirlim$, it is enough to find a subextension $K_S^R/\cL/K$ with 
\begin{equation} \label{eq:dirlimvan}
\dirlim_{\cL/L/K, \cores^{\vee}} \Sha^1(K_{S \cup T}/L, S \sm T; \mu_p)^{\vee} = 0. 
\end{equation}

Finally remark that for any subfields $K_S^R(\fc_{\leq \ell})/L^{\prime}/L/K$ the restriction maps 
\[ \res_L^{L^{\prime}} \colon \coh^1(K_{S \cup T}/L, \mu_p) \har \coh^1(K_{S \cup T}/L^{\prime}, \mu_p), \]
\noindent are injective, as one sees from the Hochschild-Serre spectral sequence using the fact that $\mu_p$ is not trivialized by $L^{\prime}$. We can and will see these restriction maps as embeddings and identify the first group with a subgroup of the second via $\res_L^{L^{\prime}}$.

%************************************************************************************************************************************************************
%************************************************************************************************************************************************************

\subsection{Construction of the tower $\cL/K$.} \label{constr_of_tower}

Recall that $\bar{\sigma} \neq 1$ denotes the image of $\sigma \in \Gal_{M/K}$ in $\Gal_{K(\mu_p)/K} = \Delta$ and $H_0 = \Gal_{K_0/K}$ is cyclic of order $d$. By the order of a character of a group we mean the cardinality of its image.

\begin{lm} There is a character $\chi \colon H_0 \rar \bF_p^{\ast}$ of order $\geq \ord(\bar{\sigma})$ and a tower of Galois extensions
\[ K \subset K_0 \subset K_1 \subset \dots \subset K_i \subset \dots \subset \bigcup\nolimits_{i=0}^{\infty} K_i =: \cL \subseteq K_S^R \] 
\noindent such that for all $i \geq 1$ we have:
\[H_i := \Gal_{K_i/K} \cong H_0 \ltimes (\prod_{j=1}^i \bZ/p\bZ), \]

\noindent where $H_0$ acts diagonally on $\prod_{j=1}^i \bZ/p\bZ$ and the action on each component is given by $\chi$.
\end{lm}

\begin{proof}
$K_0$ and $H_0$ were constructed in step 5 of Section \ref{sec:red_steps}. We have $M \cap K_S^R = K$ and hence
\[ P_{M/K}(\sigma) = \bigcup_{x \in H_0} P_{MK_0/K}(\sigma,x). \] 

\noindent up to finitely many ramified primes (cf. \cite{Wi} Proposition 2.1). By looking at the Dirichlet density, $S \cap P_{MK_0/K}(\sigma,x)$ is infinite for any $x \in H_0$, hence also $S_{K_0} \cap P_{MK_0/K}(\sigma,x)_{K_0}$ is infinite. Choose such an $x$ with $\ord(x) = \ord(\bar{\sigma})$ and write $S^{\prime} := S \cap P_{MK_0/K}(\sigma,x)$. Then for almost all $\fp \in S_{K_0}^{\prime}$, the local extensions $K_{0,\fp}/K_{\fp}$ and $K(\mu_p)_{\fp}/K_{\fp}$ are unramified of degree $\ord(\bar{\sigma})$, hence $K_0(\mu_p)_{\fp}/K_{0,\fp}$ is completely split in $\fp$, i.e., $\mu_p \subseteq K_{0,\fp}$. In particular, by \cite{NSW} 10.7.3, $X := \coh^1(K_{0, S^{\prime}}/K_0, \bZ/p\bZ)$ is infinite. $X$ is a (semisimple) $\bF_p[H_0]$-module, hence it decomposes into isotypical components $X(\phi)$ where $\phi$ goes through all $\bF_p^{\ast}$-valued characters of $H_0$. From the Hochschild-Serre spectral sequence for the Galois groups of the extensions $K_{0,S^{\prime}}/K_0/K_0^{\ker(\phi)}$ and $(\
abs{\ker(\phi)}, p) = 1$ it follows that
\begin{equation} \label{eqn:Xphi} 
X(\phi) \subseteq \coh^1((K_0^{\ker(\phi)})_{S^{\prime}}/K_0^{\ker(\phi)}, \bZ/p\bZ) \subseteq \coh^1(K_{0,S^{\prime}}/K_0, \bZ/p\bZ)
\end{equation}

\noindent In particular, if $\ord(\phi) < \ord(\bar{\sigma})$, then the order of the image of $x$ in $H_0/\ker(\phi) = \Gal_{K_0^{\ker(\phi)}/K}$ is $< \ord(\bar{\sigma})$ and for all primes $\fp \in S^{\prime}(K_0^{\ker(\phi)})$ one has $\mu_p \not\subseteq (K_0^{\ker(\phi)})_{\fp}$. By \cite{NSW} 10.7.3, the group in the middle of \eqref{eqn:Xphi} is finite and hence there must be a chracter $\chi$ of $H_0$ of order $\geq \ord(\bar{\sigma})$ such that $X(\chi)$ is infinite. For a family $(\alpha_i)_{i=1}^{\infty}$ of linearly independent elements of $X(\chi)$, let $K_0(\alpha_i)$ be the cyclic $\bZ/p\bZ$-extension of $K_0$ corresponding to $\alpha_i$ and define $K_i$ to be the compositum of the fields $\{ K_0(\alpha_j) \}_{j = 0}^i$. 
\end{proof}

%************************************************************************************************************************************************************
%************************************************************************************************************************************************************

\subsection{Action of $\Delta \times H_i$ on $\Sha^1(K_{S \cup T}/K_i, S \sm T, \mu_p)$.} \label{sec:action_of_delta_H_i}

Let $\cL/K_i/K$ be one of the fields defined above. We write 
\[ \Sha^1_i := \Sha^1(K_{S \cup T}/K_i, S \sm T; \mu_p).\] 

\noindent We have the following embeddings:

\centerline{
\begin{xy}
\xymatrix{
\Sha^1_i \ar@{^{(}->}[r] & \coh^1(K_{S \cup T}/K_i, \mu_p) \ar@{^{(}->}[r] \ar@{^{(}->}[d]^{\Delta} & \coh^1(\overline{K}/K_i, \mu_p) \ar@{=}[r] \ar@{^{(}->}[d]^{\Delta} & K_i^{\ast}/p \ar@{^{(}->}[d]^{\Delta} \\
& \coh^1(K_{S \cup T}/K_i(\mu_p), \mu_p) \ar@{^{(}->}[r] & \coh^1(\overline{K}/K_i(\mu_p), \mu_p) \ar@{=}[r] & K_i(\mu_p)^{\ast}/p
}
\end{xy}
} 

\noindent where the $\Delta$ on the arrows means that the upper entry is obtained from the lower one by taking $\Delta$-invariants. The horizontal isomorphisms on the right are canonical and given by Kummer theory. The vertical maps come from the Hochschild-Serre spectral sequence. As a subset of the lower right entry $\Sha^1_i$ defines by Kummer theory a $p$-primary Galois extension of $K_i(\mu_p)$. Further, the subgroup $\Sha^1_i$ is invariant under the $\Delta \times H_i$-action on the lower entries. Indeed, the $H_i$-invariance results simply from the definition of $\Sha^1_i$ and the fact that $S \sm T$ is defined over $K$, and the $\Delta$-invariance is obvious from the diagram. Let $L_i$ denote the abelian $p$-primary extension of $K_i(\mu_p)$, which is associated to $\Sha^1_i \subseteq K_i(\mu_p)^{\ast}/p$ via Kummer theory. The invariance discussed above implies that the composite extension $L_i/K_i(\mu_p)/K$ is Galois. Fix a trivialization of $\mu_p$; this gives an isomorphism of the Galois group of 
$L_i/K_i(\mu_p)$ with $\Sha^{1,\vee}_i := \Hom(\Sha^1_i, \bZ/p\bZ)$ and $\Delta$ acts on it via the embedding $\Delta \har \bF_p^{\ast}$. Here is a diagram of the involved extensions:

\centerline{
\begin{xy}
\xymatrix{
					& L_i \ar@{-}[d]^{\Sha^{1,\vee}_i} \\
K_i \ar@{-}[d]_{H_i} \ar@{-}[r] 	& K_i(\mu_p) \ar@{-}[d] \\
K   \ar@{-}[r]_{\Delta}			& K(\mu_p)
}
\end{xy}
} 

\noindent We have shown the following lemma:

\begin{lm}
The composite extension $L_i/K_i(\mu_p)/K$ is Galois. In particular, we have the extension of Galois groups:
\[1 \rar \Sha^{1,\vee}_i \rar \Gal_{L_i/K} \rar \Delta \times H_i \rar 1. \]
The group $\Delta \times H_i$ acts on $\Sha^{1,\vee}_i$ as follows: $\Delta$ acts by scalars via the canonical embedding $\Delta \har \bF_p^{\ast}$, and the action of $H_i$ on $\Sha^{1,\vee}_i$ is dual to the natural action of $H_i$ on $\Sha^1_i$.
\end{lm}

Observe that by construction, $L_i/K_i(\mu_p)$ is completely split in $S \sm T$. Now we investigate the action of $H_i$ more precise. For all $i > 0$, choose compatible sections $\lambda_i \colon H_0 \har H_i$ of the projections $H_i \tar H_0$ (they exist as $\abs{H_0} = d$ is prime to $[K_i:K_0] = p^i$). Via  $\lambda_i \colon H_0 \stackrel{\sim}{\rar} \lambda_i(H_0)$ we identify the character group $H_0^{\vee}$ of $H_0$ with that of $\lambda_i(H_0)$. We have a decomposition
\[ \Sha^1_i = \bigoplus_{\psi \in H_0^{\vee}} \Sha^1_i(\psi), \]

\noindent such that $\lambda_i(H_0)$ acts on $\Sha^1_i(\psi)$ by $\psi$. Observe that the subspace $\Sha^i(\psi)$ is again $\Delta \times H_i$-stable, hence the corresponding Kummer subextension $L_i(\psi)/K_i(\mu_p)$ of $L_i/K_i(\mu_p)$ is Galois over $K$. We denote the Galois group of $L_i(\psi)/K_i$ by $\Sha^1_i(\psi)^{\vee}$. We have $\Sha^1_i(\psi)^{\vee} = \Hom(\Sha^1_i(\psi),\mu_p)$ and $\lambda_i(H_0)$ acts on it by $\psi^{-1}$.

% We have a canonical identification
% \[\Sha^1_i = \{a \in K^{\ast} \colon a \in K_{\fp}^{\ast,p} \text{ for } \fp \in S \sm T; v_{\fp}(a) \equiv 0 \mod p \text{ for } \fp \not\in S \cup T \}/K^{\ast,p} \]

% \noindent Any element in $a \in \Sha^1_i$ defines an extension of $K_i(a^{1/p})/K_i$ of degree $p$, which is not Galois, but with Galois closure $K_i(\mu_p,a^{1/p})/K_i$. If $a,b \in \Sha^1_i$ are linearly independend, then the corresponding extensions $K_i(a^{1/p}), K_i(b^{1/p})$ are linearly disjoint over $K_i$ and $K_i(\mu_p,a^{1/p}),K_i(\mu_p,b^{1/p})$ are linearly disjoint over $K_i(\mu_p)$.

%************************************************************************************************************************************************************
%************************************************************************************************************************************************************

\subsection{Reduction to uniform boundedness} \label{sec:red_to_unif_boundedness}

We reduce equation \eqref{eq:dirlimvan} for the tower $\cL/K$ defined in Section \ref{constr_of_tower} which we have to show, to the following two propositions (which we will prove in Subsections \ref{sec:deal_with_chars_of_ord} and \ref{sec:unif_bouns_via_infdens}), both of them bounding $\Sha^1_i(\psi)$ in two different cases:

\begin{prop}\label{prop:cont_in_H_0_for_K_0}
Let $i \geq 1$ and let $\psi \in H_0^{\vee}$ be of order $< \ord(\bar{\sigma})$. Then 
\[ \Sha_i^1(\psi) \subseteq \coh^1(K_{S \cup T}/K_0, \mu_p)\] 
\noindent (both regarded as subgroups of $\coh^1(K_{S \cup T}/K_i,\mu_p)$).  
\end{prop}

\begin{prop}\label{prop:unif_bound} There is a constant $C > 0$ depending only on $M/K,p, \sigma$ (but not on $i$) such that for all $\psi \in H_0^{\vee}$ of order $\geq \ord(\bar{\sigma})$ one has
\[\abs{\Sha^1_i(\psi)} < C \]

\noindent for each $i \geq 1$.
\end{prop}

Indeed, to deduce equation \eqref{eq:dirlimvan}, it is enough to show that for $j \gg i \gg 0$, the map 
\[ \cores_{ji} \colon \Sha^1_j \rar \Sha^1_i \]
is the zero map (we denote by $\cores_{ji}$ resp. $\res_{ij}$ the corestriction resp. the restriction maps between the levels $K_i$ and $K_j$ for $i \leq j$). By compatibility of the chosen sections $\lambda_i \colon H_0 \har H_i$, we have $\res_{ij}(\Sha_i^1(\psi)) \subseteq \Sha_j^1(\psi)$ for $i \leq j$. Since the restriction maps are injective, we can choose by Proposition \ref{prop:unif_bound} an $i_0 \geq 0$ such that the inclusion 
\[\res_{ij} \colon \Sha_i^1(\psi) \har \Sha_j^1(\psi) \]

\noindent is an isomorphism for all $j \geq i \geq i_0$ and all $\psi \in H_0^{\vee}$ of order $\geq \ord(\bar{\sigma})$. Then for $j > i \geq i_0$ we have:

\[ \Sha_j^1 = \bigoplus_{\substack{\psi \in H_0^{\vee} \\ \ord(\psi) \geq \ord(\bar{\sigma})}} \Sha_j^1(\psi) \oplus \bigoplus_{\substack{\psi \in H_0^{\vee} \\ \ord(\psi) < \ord(\bar{\sigma})}} \Sha_j^1(\psi), \]

\noindent where the first summand is contained in $\res_{i_0 j}(\Sha_{i_0}^1)$ and the second summand is contained in $\coh^1(K_{S \cup T}/K_0, \mu_p)$ by Proposition \ref{prop:cont_in_H_0_for_K_0}. Thus we conclude that 
\[\Sha_j^1 \subseteq \res_{ij}(\coh^1(K_{S \cup T}/K_i, \mu_p)).\]

\noindent where both groups are seen as subgroups of $\coh^1(K_{S \cup T}/K_j,\mu_p)$. Finally, recall that for $L^{\prime}/L$ Galois, the composition $\cores_L^{L^{\prime}} \circ \res^L_{L^{\prime}}$ is equal to the multiplication by the degree $[L^{\prime}:L]$, that further $p$ divides $[K_j:K_i]$ for $j > i \geq 0$ and that the group $\coh^1(K_{S \cup T}/K_i, \mu_p)$ is killed by $p$. So, for all $i,j$ with $j > i \geq i_0$ and for any $a \in \Sha_j^1$ with preimage $b \in \coh^1(K_{S \cup T}/K_i, \mu_p)$ we have 
\[ \cores_{ji}(a) = \cores_{ji}\res_{ij}(b) = 0, \]

\noindent i.e., $\cores_{ji} \colon \Sha^1_j \rar \Sha^1_i$ is the zero map. Hence also $\cores_{ji}^{\vee}$ is the zero map, which shows equation \eqref{eq:dirlimvan}.

%************************************************************************************************************************************************************
%************************************************************************************************************************************************************

\subsection{Dealing with characters of small order}\label{sec:deal_with_chars_of_ord}

Here is the proof of Proposition \ref{prop:cont_in_H_0_for_K_0}: let $K_0^{\psi} := (K_0)^{\ker(\psi)}$, which is a proper subfield of $K_0$. Let $H_i^{\psi} := \Gal_{K_i/K_0^{\psi}} = \pi_i^{-1}(\ker(\psi))$, where $\pi_i$ denotes the projection $H_i \tar H_0$. Further, $\lambda_i(\ker(\psi))$ acts trivially on $\Sha^1_i(\psi)$. With $\lambda_i(\ker(\psi))$ also the normal subgroup $\langle \langle \lambda_i(\ker(\psi)) \rangle \rangle$ generated by it in $H_i$ acts trivially on $\Sha^1_i(\psi)$. By Lemma \ref{lm:normteileristgross}, $\langle \langle \lambda_i(\ker(\psi)) \rangle \rangle = H_i^{\psi}$. Hence:

\[ \Sha^1_i(\psi) \subseteq \coh^1(K_{S \cup T}/K_i, \mu_p)^{H_i^{\psi}} = \coh^1(K_{S \cup T}/K_0^{\psi}, \mu_p), \]

\noindent where the last equality (inside $\coh^1(K_{S \cup T}/K_i, \mu_p)$) results from the Hochschild-Serre spectral sequence and the fact that $\mu_p(K_i) = \{ 1 \}$. Finally, Proposition \ref{prop:cont_in_H_0_for_K_0} follows as $\coh^1(K_{S \cup T}/K_0^{\psi}, \mu_p) \subseteq \coh^1(K_{S \cup T}/K_0, \mu_p)$ via restriction.

\begin{lm}\label{lm:normteileristgross}
We have $\langle \langle \lambda_i(\ker(\psi)) \rangle \rangle = H_i^{\psi}$.
\end{lm}
\begin{proof}
We can represent $H_i$ as the follows (recall that $\chi \colon H_0 \rar \bF_p^{\ast}$ is the character defining the action of $H_0$ on $\ker(H_i \tar H_0)$; it has order $\geq \ord(\bar{\sigma})$):
\[H_i \cong \left\{ (a,v) \colon a \in H_0, v \in \bF_p^i \right\}, \quad (a,v).(b,w) = (ab,v + \chi(a)w). \]

\noindent As $\ord(\chi) \geq \ord(\bar{\sigma}) > \ord(\psi)$, we have $\ker(\psi) \supsetneq \ker(\chi)$. Let $h \in \ker(\psi) \sm \ker(\chi)$. Write $\lambda_i(h) = (h,v)$. Then for any $w \in \bF_p^i$, the commutator
% \[ \matzz{a}{v}{0}{1}^{-1} \matzz{1}{w}{0}{1} \matzz{a}{v}{0}{1} \matzz{1}{-w}{0}{1} = \matzz{1}{a^{-1}w - w}{0}{1} \]
\[ (h,v)^{-1}.(1,w).(h,v).(1,-w)  = (1,\chi(h)^{-1}w - w) \]
\noindent lies in $\langle \langle \lambda_i(\ker(\psi)) \rangle \rangle$. As $1 \neq \chi(h) \in \bF_p^{\ast}$, we easily see that $\langle \langle \lambda_i(\ker(\psi)) \rangle \rangle = H_i^{\psi}$. 
\end{proof}

%************************************************************************************************************************************************************
%************************************************************************************************************************************************************

\subsection{Uniform bounds and generalized densities} \label{sec:unif_bouns_via_infdens}

It remains to prove Proposition \ref{prop:unif_bound}. We use the fixed Frobenius densities introduced in preceding sections. All densities are taken over $K$, so we omit $K$ from the notation and write $\delta_{L,x}$ instead of $\delta_{L/K,x}$ if $L/K$ is finite Galois and $x \in \Gal_{L/K}$. Let $S_0 := P_{M/K}(\sigma) \cap S$. Then $S_0 \backsimeq P_{M/K}(\sigma)$. For any $i > 0$ and any $x \in H_i$, we consider the element $(\bar{\sigma},x) \in \Delta \times H_i = \Gal_{K_i(\mu_p)/K}$. We apply Corollary \ref{cor:counting_inf_dens_of_Chebotarevs} to $\sigma \in \Gal_{M/K}$ and $(\bar{\sigma},x) \in \Delta \times H_i$: $\sigma$ and $(\bar{\sigma},x)$ lie over the same element $\bar{\sigma} \in \Delta$ and $M \cap K_i(\mu_p) = K(\mu_p)$. Hence

\begin{eqnarray} 
\nonumber \delta_{K_i(\mu_p), (\bar{\sigma},x)}(S_0) &=& \frac{ \abs{C((\sigma,x), \Gal_{MK_i/K})} }{[M:K(\mu_p)] \abs{ C((\bar{\sigma},x), \Delta \times H_i) } } \\ 
 &=&  \frac{ \abs{C(\sigma, \Gal_{M/K})} \abs{C(x, H_i)} }{ [M:K(\mu_p)] \abs{ C(\bar{\sigma}, \Delta) } \abs{ C(x, H_i) } } \\ 
\nonumber &=&  \frac{ \abs{C(\sigma, \Gal_{M/K})} }{ [M:K(\mu_p)] \abs{ C(\bar{\sigma}, \Delta) } } 
\end{eqnarray}

\noindent (this computation uses that all involved Galois groups which are a priori fibered products, decompose into simple direct products). Thus we see that for any $x$, the $(\bar{\sigma},x)$-density of $S_0$ in $K_i(\mu_p)$ remains constant $> 0$ and independent of $i$ and of $x$. Let $C > 0$ be some fixed constant such that 

\begin{equation}\label{eq:densbound_from_below} 
\delta_{K_i(\mu_p), (\bar{\sigma},x)}(S_0) > C^{-1}. 
\end{equation}

Let now $x \in H_0$ be an element such that $\psi(x) = \bar{\sigma} \in \Delta \subseteq \bF_p^{\ast}$. This choice is possible since $\ord(\psi) \geq \ord(\bar{\sigma})$ and hence $\langle \bar{\sigma} \rangle \subseteq \psi(H_0) \subseteq \Delta \subseteq \bF_p^{\ast}$ (being cyclic, $\bF_p^{\ast}$ has at most one subgroup of each order). Thus the element $y := (\bar{\sigma},\lambda_i(x)) \in \Delta \times H_i$ operates on $\Sha^1_i(\psi)^{\vee}$ trivially. Consider the Galois extensions:

\centerline{
\begin{xy}
\xymatrix{
L_i(\psi) \ar@{-}[d]^{\Sha^1_i(\psi)^{\vee}} \\
K_i(\mu_p) \ar@{-}[d]^{\Delta \times H_i} \\
K	
}
\end{xy}
} 

\noindent We have the following commutative diagram with exact rows:

\centerline{
\begin{xy}
\xymatrix{ 
1 \ar[r] & \Sha^1_i(\psi)^{\vee} \ar[r] & \Gal_{L_i(\psi)/K} \ar[r]^{\pi} & \Delta \times H_i \ar[r] & 1 \\
1 \ar[r] & \Sha^1_i(\psi)^{\vee} \ar[r] \ar@{=}[u] & G_y \ar[r] \ar@{^{(}->}[u] & \langle y \rangle \ar[r] \ar@{^{(}->}[u] & 1, \\
}
\end{xy}
} 

\noindent where $G_y$ is defined to be the pull-back of $\langle y \rangle$ and $\Gal_{L_i(\psi)/K}$ over $\Delta \times H_i$. Now $\ord(\bar{\sigma})|\ord(x) = \ord(\lambda_i(x))$. Hence $\ord(y) = \lcm(\ord(\bar{\sigma}),\ord(x)) = \ord(x)$ is coprime to $p$. The group $\Sha^1_i(\psi)^{\vee}$ is abelian $p$-primary, hence the lower sequence in the above diagram splits. Since by construction the action of $y$ on $\Sha^1_i(\psi)^{\vee}$ is trivial, we have: $G_y \cong \Sha^1_i(\psi)^{\vee} \times \langle y \rangle$. This shows explicitely that there is precisely one element $\tilde{y}$ in the preimage of $y$ in $G_y$ (resp. in $\Gal_{L_i(\psi)/K}$, which is the same) such that $\ord(\tilde{y}) = \ord(y)$.

As in Section \ref{sec:pull-back-properties}, for $z \in \pi^{-1}(y)$, let $M_z$ be the image of $P_{L_i(\psi)/K}^z$ in $P_{K_i(\mu_p)/K}^y$ under the natural projection map. In particular, Proposition \ref{prop:verallgCheb} gives

\begin{equation}\label{eq:compofdens_of_M_y_tilde}
\delta_{K_i(\mu_p),y}(M_{\tilde{y}}) =  \frac{1}{\abs{\Sha^1_i(\psi)^{\vee}}} = \frac{1}{\abs{\Sha^1_i(\psi)}}
\end{equation}

\noindent as by the above order computations, the $\Sha^1_i(\psi)^{\vee}$-conjugacy class of $\tilde{y}$ in $\pi^{-1}(y)$ contains the only element $\tilde{y}$. The fundamental observation is now the following lemma.

\begin{lm}\label{lm:_inf_dens_of_our_cs_set}
We have $P_{K_i(\mu_p)/K}^y \cap \cs(L_i(\psi)/K_i(\mu_p)) \subseteq M_{\tilde{y}}$.
\end{lm}
\begin{proof}
Let $\fp \in P_{K_i(\mu_p)/K}^y \cap \cs(L_i(\psi)/K_i(\mu_p))$. Then $\fp$ is unramified in $L_i(\psi)/K_i(\mu_p)$ and hence lies in one of the sets $M_z$ for some $z \in \pi^{-1}(y)$. Thus the Frobenius of a lift of $\fp$ to $L_i(\psi)$ is $\Sha^1_i(\psi)^{\vee}$-conjugate to $z$ inside $\pi^{-1}(y)$. But since $\fp$ is completely split in $L_i(\psi)$, we must have $\ord(z) = \ord(y)$, and this can only be satisfied for $z = \tilde{y}$. 
\end{proof}

Finally, $S_0 \sm T \subseteq \cs(L_i(\psi)/K_i(\mu_p))$ by construction and Lemma \ref{lm:_inf_dens_of_our_cs_set} implies 
\[(S_0 \sm T) \cap P_{K_i(\mu_p)/K}^y = (S_0 \sm T) \cap M_{\tilde{y}}. \] 
Together with \eqref{eq:compofdens_of_M_y_tilde} and Corollary \ref{cor:bc_of_delta_x}, this gives (since $T$ is finite, we can ignore it in density computations):

\begin{eqnarray*}
1 &\geq& \delta_{L_i(\psi), \tilde{y}} (S_0) \\ 
&=& \delta_{K_i(\mu_p), y} (M_{\tilde{y}})^{-1} \delta_{K_i(\mu_p),y}(S_0 \cap M_{\tilde{y}}) \\
&=& \abs{\Sha^1_i(\psi)} \delta_{K_i(\mu_p),y}(S_0 \cap P_{K_i(\mu_p)/K}^y) \\
&=& \abs{\Sha^1_i(\psi)} \delta_{K_i(\mu_p),y}(S_0).
\end{eqnarray*}

\noindent Hence by \eqref{eq:densbound_from_below}:
\[ \abs{\Sha^1_i(\psi)} \leq \delta_{K_i(\mu_p),y}(S_0)^{-1} < C. \]

\noindent This finishes the proof of Proposition \ref{prop:unif_bound} and hence of Theorem \ref{thm:real_of_loc_ext}.

%************************************************************************************************************************************************************
%************************************************************************************************************************************************************

\renewcommand{\refname}{References}

%************************************************************************************************************************************************************
%************************************************************************************************************************************************************

%************************************************************************************************************************************************************
%************************************************************************************************************************************************************

%************************************************************************************************************************************************************
%************************************************************************************************************************************************************

%************************************************************************************************************************************************************
%************************************************************************************************************************************************************

\end{document}